\documentclass{article}
\usepackage[utf8]{inputenc}
\usepackage{amssymb,amsthm}
\usepackage{amsmath}
\usepackage{abraces}
\usepackage[all]{xy}
\usepackage{enumerate}
\usepackage{todonotes} 
\usepackage{tablefootnote} 
\usepackage{mathtools} 
\usepackage{subfigure}

\newcommand{\R}{\mathbb{R}}
\newcommand{\tanDelta}{T_{\bf{e_a}}\Delta_{N}^0}
\newcommand{\simplex}{\Delta_N^0}
\newcommand{\Expa}{\mathrm{Exp}_a}
\newcommand{\Loga}{\mathrm{Log}_a}
\newcommand{\Log}{\mathrm{Log}}
\newcommand{\Exp}{\mathrm{Exp}}
\newcommand{\plus}{\oplus}
\newcommand{\ori}{\odot}
\newcommand{\orifa}{\odot_{_{\fa }}}
\newcommand{\plusfa}{\oplus_{_{\fa}}}
\newcommand{\minusfa}{\ominus_{\fa}}
\newcommand{\fa}{\mathcal{C}_{\bf a }}

\newcommand{\fu}{\mathcal{C}_{\bf 1 }}
\newcommand{\fap}{\mathcal{C}_{\bf a'}}
\newcommand{\Logfa}{\mathrm{Log}_{_{\fa }}}
\newcommand{\Expfa}{\mathrm{Exp}_{_{\fa}}}
\newcommand{\plusunu}{\oplus_{_{\fu }}}
\newcommand{\oriunu}{\odot_{_{\fu}}}
\newcommand{\Logfunu}{\mathrm{Log}_{_{\fu}}}
\newcommand{\Expfunu}{\mathrm{Exp}_{_{\fu}}}
\newcommand{\distfa}{\mathrm{d}_{_{\fa}}}
\newcommand{\distfunu}{\mathrm{d}_{_{\fu}}}
\newcommand{\alg}{\mathrm{A}_{k}}
\newcommand{\sub}{\mathrm{Sub}}

\newcommand{\bl}{\boldsymbol\lambda}
\newcommand{\bx}{{\bf x}}
\newcommand{\tax}{t_{\bf a}({\bf x})}

\DeclareMathOperator*{\argmin}{arg\min}
\DeclarePairedDelimiter\ps{\langle}{\rangle}
\DeclarePairedDelimiter\psa{\langle}{\rangle_{\fa}}

\newtheorem{lemma}{Lemma}
\newtheorem{theorem}{Theorem}

\newtheorem{definition}{Definition}
\newtheorem{proposition}{Proposition}

\title{Parametric Lie group structures on the probabilistic simplex and generalized Compositional Data}
\date{\today}
\author{ Petre Birtea$^1$, Ioana Gavra$^{2}$   
\\  $^{1}$Departement of Mathematics, West University of Timișoara \\ $^{2}$  IRMAR (UMR CNRS 6625), Universit\'e de Rennes }

\begin{document}

	\maketitle
	
	\begin{abstract}
In this paper we build a set of parametric quotient Lie group structures on the probabilistic simplex that can be extended to real vector space structures. In particular, we rediscover the main mathematical objects generally used when treating compositional  data as elements associated to the quotient Lie group with respect to the equivalence relation induced by the scale invariance principle. This perspective facilitates the adaptation of the statistical methods used for classical compositional data to data that follows a more general equivalence relation. 
	
	\end{abstract}
	\section{Introduction}
	
The unit simplex can be regarded as the set of discrete probability distributions for random variables with a fixed finite number of possible values. It is also used as a model space for Compositional Data analysis (\cite{Aitchinson2001},\cite{GreenacreBook}), for word frequencies in text classification (\cite{textClass}), portfolio selection problems (\cite {port-sel}) and a myriad of other problems.
From a geometric point of view, it can be regarded as a constrained sub-manifold of the open half-cone.

	The open half-cone $\mathbb{R}^{N+1}_{>0}=\left\{{\bf x}\in \mathbb{R}^{N+1}|~ x_{i}>0\right\}$ can be endowed with an Abelian Lie group structure
	\begin{equation*}
	{\bf x}\plus {\bf y}=(\dots,x_iy_i,\dots),
	\end{equation*}
	with neutral element ${\bf 1}=(\dots,1,\dots)\in \mathbb{R}^{N+1}_{>0}$.
	
	The Lie group exponential map is given by the {\bf global diffeomorphism} $\Exp:\mathbb{R}^{N+1} \rightarrow \mathbb{R}^{N+1}_{>0}$,
	
	\begin{equation*}
	\Exp({\boldsymbol \xi})=(\dots,e^{\xi_{i}},\dots)
	\end{equation*}
	and its inverse, the logarithmic map  ${\bf \Log}:\mathbb{R}^{N+1}_{>0}\rightarrow\mathbb{R}^{N+1}$,
	
	\begin{equation*}
	{\Log}({\bf x})=(\dots,\ln x_i,\dots).
	\end{equation*}
	
	Using the global diffeomorphism property of ${\bf \Exp}$ and ${\bf \Log}$, we can introduce a scalar multiplication on $\mathbb{R}^{N+1}_{>0}$ as follows,
	
	\begin{equation}\label{inmultire cu scalar}
	c\ori {\bf x}:=\Exp(c {\Log}(\bf x)).
	\end{equation}
	Writing in coordinates, we have $$c\ori {\bf x}=(\dots,{x_i}^c,\dots),$$ for all $c\in \mathbb{R}$. The set $\mathbb{R}^{N+1}_{>0}$ endowed these two operations becomes a real vector space
		$(\mathbb{R}^{N+1}_{>0}, \plus, \ori).$

		The probabilistic open simplex 
		$$ \simplex=\{ {\bf x}\in \R^{N+1}_{>0}\ | \sum_{i=1}^{N+1} x_i =1 \}$$
	is a constraint manifold, which also has an Abelian group structure :
	
		\begin{equation*}
	{\bl} * {\boldsymbol{\mu}}=\frac{1}{\sum_{i=1}^{N+1} \lambda_i\mu_i}(\dots,\lambda_i \mu_i,\dots).
	\end{equation*}
	This group operation, also called perturbation, lays at the core of the Compositional Data analysis methods. It was also used together with the Fisher Riemannian metric on the simplex in the problem of metric learning for text  classification in \cite{textClass}.\\
	
	A first observation is that $(\simplex,*)$ is not a subgroup of	$(\mathbb{R}^{N+1}_{>0}, \plus).$ A natural question that arises is {\bf how can one geometrically explain the group operation $*$ in relation with the group operation $\oplus$}. 
	
	We prove that $(\simplex,*)$ can be regarded as a quotient Lie group structure of $(\mathbb{R}^{N+1}_{>0}, \plus)$ relative to the one-dimensional subgroup $H=\left\{\Exp(t{\bf 1})=\right.$ $\left. (\dots,e^{t},\dots)|~  t\in \mathbb{R}\right\}$. Furthermore, this structure can be extended to a vector space one.  This allows, for the first time to give a Lie group explanation for several of the main objects encountered in Compositional Data analysis. For example, the Centered-Log ratio transformation is non other than the Lie logarithmic map of the quotient structure $(\simplex,*)$ (for details, see sub-section \ref{sec: composition}). It also facilitates a Lie group interpretation to the famous Boltzmann distribution, also known as softmax function in the machine learning community.\\
	In section \ref{sec: 2}, we give a follow up of the above constructions to the case where the Lie group structure and the associated objects depend on a parameter. In particular we obtain a parametric family of distances on the simplex. Parametric families of distances are an important tool in metric learning problems (\cite{metricLearningBrian}, \cite{metricLearningXing}).  Our idea is to embed this parameter dependence into the family of one-dimensional subgroups, with respect to which we construct the quotient structure on the probabilistic simplex.
	
	In Section \ref{sec : 3}, we detail the link between the quotient Lie group structure induced by $H$ and the Aitchison geometry. Sometimes the nature of the data might require a different quotient geometry, like for example data that comes from exponential decay phenomena \cite{usesandmisuses}. In    Section \ref{sec:4}, we extend the main principles of Compositional Data analysis to this more general framework. Finally, in Section \ref{sec:5}, we present the adaptation of some basic objects of descriptive statistics and of the multivariate normal distribution to the probabilistic simplex endowed with this new geometric structures.

		\section{The probabilistic simplex as a quotient Lie group}\label{sec: 2}
	From general theory of Lie groups,  we have that any one-dimensional subgroup is of type

	\begin{equation}\label{defH}
	H_{\bf a}=\left\{\Exp(t{\bf a})=(\dots,e^{ta_i},\dots)|~  t\in \mathbb{R}\right\},\quad \mbox{ with } {\bf a}\in \mathbb {R}^{N+1}.
	\end{equation}
	It is easy to see that those are also all the one-dimensional linear subspaces  of $(\mathbb{R}^{N+1}_{>0}, \plus, \ori)$. 
	The one-dimensional subgroup $H_{\bf a}$ is a normal subgroup of the Lie group $(\mathbb{R}^{N+1}_{>0},\plus)$ and consequently, we obtain the quotient Lie group $\mathbb{R}^{N+1}_{>0}/H_{\bf a}$, with $\pi_{\bf a}:\mathbb{R}^{N+1}_{>0}\rightarrow\mathbb{R}^{N+1}_{>0}/H_{\bf a}$ being the projection map.
	
	The quotient group structure is defined by 
	
	$$[\bf x]_{\bf a}\diamond_{\bf a}[\bf y]_{\bf a}:=\pi_{\bf a}({\bf x}\plus{\bf y}).$$
	and we have the following commutative diagram: 
	
	$$ \xymatrix{
		\mathbb{R}^{N+1}_{>0}  \ar@<1ex>[rr]^*{\Log} \ar[d]^*{\pi_{\bf a}}	&& T_{\bf 1}\mathbb{R}^{N+1}_{>0}\ar[ll]^*{\Exp} \ar[d]^*{D\pi_{\bf a}({\bf 1)}} \\
		\mathbb{R}^{N+1}_{>0}/H_{\bf a} \ar@<1ex>[rr]^*{\Log_{\bf a}} && T_{[{\bf 1}]}\left( \mathbb{R}^{N+1}_{>0}/H_{{\bf a}} \ar[ll]^*{\Exp_{\bf a}}\right) }$$
	where  $ T_{[{\bf 1}]}\left( \mathbb{R}^{N+1}_{>0}/H_{{\bf a}}\right)$ is the Lie algebra corresponding to the quotient Lie group $\mathbb{R}^{N+1}_{>0}/H_{\bf a}$ . The maps $\Loga$ and $\Expa $ are the exponential and logarithmic maps associated with the quotient group structure. In this case they are also global diffeomorphisms.
	Furthermore, the quotient group $\mathbb{R}^{N+1}_{>0}/H_{\bf a}$ can be endowed with a scalar multiplication as in (\ref{inmultire cu scalar}) and the corresponding vector structures are compatible with  $\Loga$ and $\Expa $, in the sense that these two become linear maps.
	
	\medskip
	
	Next, we show that for certain choices of $\bf a$, the simplex $\Delta_{N}^0$  \textbf {can be regarded as a model space for the quotient group} $(\mathbb{R}^{N+1}_{>0}/H_{\bf a},\diamond_{\bf a})$. 
	
	The classes defined by the equivalence relation induced by the subgroup $H_{\bf a}$ are the smooth curves ${\bf{x}} H_{\bf a}=\left\{{\bf x}\plus \Exp(t{\bf a})=(\dots,x_{i}e^{ta_i},\dots)|~  t\in \mathbb{R}\right\}$,  $\ {\bf x}\in \mathbb {R}^{N+1}_{>0}$ and $\pi_{\bf a}({\bf{x}} H_{\bf a})=[\bf x]_{\bf a}$.
	
	In order to conclude, we need the following result.
	
	\begin{lemma}
		Let ${\bf a} \in \mathbb{R}^{N+1}$ and $H_{\bf a}$ the associated subgroup  defined in \eqref{defH}. The following statements are equivalent:  
		\begin{itemize}
			\item $\Delta_{N}^0 $ intersects $ {\bf{x}} H_{\bf a}$ in a unique point for each ${\bf x} \in \mathbb{R}_{>0}^{N+1}$ 
			\item $a_i>0$, for all $0\le i\le N+1$ or $a_i<0$, for all $0\le i\le N+1$.
		\end{itemize}
		
	\end{lemma}
	\begin{proof}
		
		For all $\bf x \in \mathbb{R}^{N+1}_{>0}$,  
		\begin{equation*}\label{intersection}
		{\bf{x}} H_{\bf a}\cap \Delta_N^0=\{(\dots,x_{i}e^{ta_i},\dots) \ | \ t \in \mathbb{R} \mbox{ s.t. }\sum_{i=1}^{N+1} x_ie^{ta_i}=1 \}.
		\end{equation*}
		Assume that all the coordinates $a_i$ of the vector $\bf a$ are strictly positive.
		Then, for any $\bf x \in \mathbb{R}^{N+1}_{>0}$  the intersection ${\bf{x}} H_{\bf a}\cap \Delta_N^0$ is a singleton  since  $ \mathbb{R} \ni t\mapsto \sum_{i=1}^{N+1} x_ie^{a_it}$ is a strictly monotone continuous function whose image is $(0,+\infty)$. The same type of reasoning applies for $a_i<0, i=\overline{1,N+1}$.
		
		To see that this condition is also necessary to ensure that each equivalence class intersects the simplex in a unique point, first observe that if there exists $a_j=0$, then $\sum_{i=1}^{N+1}e^{a_it}>1$, for all $t \in \mathbb{R}$ and thus, for $\textbf{x} =(1,\ldots,1)$, we obtain ${\bf x}H_{\bf a}\cap \Delta_{N}^0 =\emptyset$.\\
		Next, suppose that there exists $i\ne j$ such that $a_i<0$ and $a_j>0$. Then for all $t \in \mathbb{R}^{\star}$ either $e^{a_it}>1$ or $e^{a_jt}>1$ which implies $e^{a_jt}+e^{a_it}>1$, for all $t \in \mathbb{R}.$
		Therefore, we can conclude that if there exist $i\ne j$ such that $a_i<0$ and $a_j>0$ then once more, for $\textbf{x} =(1,\ldots,1)$, we get ${\bf x}H_{\bf a}\cap \Delta_{N}^0=\emptyset $. \\
	\end{proof}

	The above Lemma allows us to define the map $\fa:\mathbb{R}^{N+1}_{>0}\rightarrow \Delta_{N}^0$, ${\bf x}\mapsto \fa({\bf x}):= {\bf x}H_{\bf a}\cap \Delta_{N}^0$,  the unique point of intersection between the class ${\bf x}H_{\bf a}$ and the open simplex $\Delta_{N}^0$. This in turn will define the bijection $s_{\bf a}:\mathbb{R}^{N+1}_{>0}/H_{\bf{a}}\rightarrow \Delta_{N}^0$ and we have the following commutative diagram
	
	$$ \xymatrix{
		\mathbb{R}^{N+1}_{>0} \ar[d]_*{\pi_{\bf{a}}} \ar[r]^*{\fa}	& \Delta_{N}^0 \subset \mathbb{R}^{N+1}_{>0}\\
		\mathbb{R}^{N+1}_{>0}/H_{\bf{a}}\ar[ru]_*{s_{\bf{a}}} & }$$
	
	We endow $\Delta_{N}^0$ with the group structure that makes the bijection $s_{\bf a}$ a homomorphism, i.e. 
	$$\boldsymbol{\lambda}_1\oplus_{\bf a}\boldsymbol{\lambda}_2:=s_{\bf a}([\boldsymbol{\lambda}_1]_{\bf a}\diamond_{\bf a} [\boldsymbol{\lambda}_2]_{\bf a})=\fa(\boldsymbol{\lambda}_1\oplus\boldsymbol{\lambda}_2),$$
	where in the right side we regard $\boldsymbol{\lambda}_1$ and $\boldsymbol{\lambda}_2$ as points in $\Delta_{N}^0\subset\mathbb{R}^{N+1}_{>0}$. In other words, to obtain the composition of two elements $\boldsymbol{\lambda}_1$ and $\boldsymbol{\lambda}_2$ in the open simplex, 
	we regard $\boldsymbol{\lambda}_1$ and $\boldsymbol{\lambda}_2$ as elements of the ambient group $(\mathbb{R}^{N+1}_{>0},\plus)$ and compute the intersection of the class $(\boldsymbol{\lambda}_1\plus\boldsymbol{\lambda}_2)H_{\bf a}$ with $\Delta_{N}^0$.
	
	\medskip 
	
	In order to render some of the above constructions in a more explicit form, we need to give the analytical construction of the map $\fa$. By definition, $\fa$ can be constructed as 
	\\
	\begin{equation}\label{def fa}
	\fa({\bf x})=h_{\bf x, \bf a}(t_{\bf a}(\bf x)),
	\end{equation}
	\\
	where the smooth curve $ h_{\bf x, \bf a}(t)=(\dots,x_{i}e^{a_it},\dots)$ defines the class ${\bf{x}} H_{\bf a}$ and $t_{\bf a}(\bf x)\in \mathbb{R}$ is the solution of the equation
	\begin{equation}\label{implicit eq}
	\sum_{i=1}^{N+1} x_{i}e^{a_it}=1.
	\end{equation}
	
	In some particular cases, we can solve the above equation for $t$ but considering a general vector ${\bf a}$ with strictly positive/negative coordinates, the above equation is impossible to solve analytically.  
	
	Nevertheless, in order to compute the exponential and logarithm map for the quotient Lie group $(\Delta_{N}^0,\plusfa)$, we need only the derivative of the homomorphism $\fa:(\mathbb{R}^{N+1}_{>0},\oplus) \longrightarrow (\Delta_{N}^0,\plusfa)$ computed in the neutral element ${\bf 1}$. More precisely,
	$$
	(D_{{\bf x}}\fa)_{i,j}({\bf 1})=\frac{\partial (x_{i}e^{a_it_{\bf a}(\bf x)})}{\partial x_j}({\bf 1})=\delta_{i,j}e^{a_it_{\bf a}(\bf 1)}+a_{i}\frac{\partial t_{\bf a}}{\partial x_{j}}({\bf 1})e^{a_it_{\bf a}(\bf 1)},
	$$
	where $e^{a_it_{\bf a}(\bf 1)}$ is the $i^{\text{th}}$ coordinate of the neutral element ${\bf e}_{{\bf a}}$ of the quotient Lie group $(\Delta_{N}^0,\plusfa)$ and $\dfrac{\partial t_{\bf a}}{\partial x_{j}}(\bf 1)$ is computed using The Implicit Function Theorem applied to equation \eqref{implicit eq} as follows, 
	
	$$
	\frac{\partial t_{\bf a}}{\partial x_{j}}({\bf 1})=-\frac{1}{\sum_{1}^{N+1}a_k({\bf e_{\bf a}})_k}({\bf e}_{\bf a})_{j}.
	$$
	
	Consequently, for the derivative of the homomorphism  $\fa:\mathbb{R}^{N+1}_{>0} \longrightarrow \Delta_{N}^0$ computed in the neutral element ${\bf 1}$, we have the formula
	
	\begin{equation}\label{formula Dx}
	(D_{{\bf x}}\fa)_{i,j}({\bf 1})=\delta_{i,j}({\bf e}_{\bf a})_{i}-\frac{a_i}{\sum_{k=1}^{N+1}a_k({\bf e_{\bf a}})_k}({\bf e}_{\bf a})_{i}({\bf e}_{\bf a})_{j}.
	\end{equation}

	\bigskip
	
	The homomorphism $\fa:\mathbb{R}^{N+1}_{>0} \longrightarrow \Delta_{N}^0$ makes the following diagram commutative, \textcolor{black}{see Appendix},

	\begin{equation}\label{diagrama fa}
	\xymatrix{
		T_{\bf 1}	\mathbb{R}^{N+1}_{>0}  \ar[rr]^*{D_{{\bf x}}\fa(1)} \ar@<1ex>[d]^*{\Exp}	&& T_{\bf{e_a}}\Delta_{N}^0 \ar[d]^*{\Expfa} \\
		(\mathbb{R}^{N+1}_{>0},\plus)\ar[u]^*{\Log} \ar[rr]^*{\fa} && (\Delta_{N}^0, \plusfa)\ar@<1ex>[u]^*{\Logfa}}
	\end{equation}
	
	\noindent where the Lie algebras are given by $T_{\bf 1}	\mathbb{R}^{N+1}_{>0}=\mathbb{R}^{N+1}$ and  $T_{\bf{e_a}}\Delta_{N}^0 =\{ \boldsymbol{ \xi} \in \mathbb{R}^{N+1} \  | \displaystyle\sum_{i=1}^{N+1} \xi_i=0 \}.$ As in the case of the Lie group $(\mathbb{R}^{N+1}_{>0},\plus)$, we obtain for the Lie group $(\Delta_{N}^0, \plusfa)$ that $\Expfa$ and $\Logfa$ are also {\bf global diffeomorphisms}.

	\medskip

	The commutativity of the above diagram leads \textcolor{black}{directly} to an explicit formula for $\Logfa$ but not for $\Expfa$.
	\begin{theorem}\label{thm: log}
		The logarithm map $\Logfa: \Delta_{N}^0 \rightarrow T_{\bf{e_a}}\Delta_{N}^0$ of the quotient Lie group structure on $\Delta_{N}^0 $ induced by the subgroup $H_{\bf a}$ has the formula :
		
		$$\Logfa(\boldsymbol{\lambda})=\left(\dots, ({\bf e}_{\bf a})_{i}\left( \ln(\lambda_i)-\frac{a_i}{\sum_{k=1}^{N+1}a_k({\bf e_{\bf a}})_k}\sum_{j=1}^{N+1}({\bf e}_{\bf a})_{j}\ln(\lambda_j)\right),\dots\right),$$
		where ${\bf e_{\bf a}}=\fa({\bf 1})$ is the neutral element of $(\Delta_{N}^0,\plusfa)$.
	\end{theorem}
	
	\begin{proof} Note that, for ${\bf \boldsymbol{\lambda}}\in \Delta_{N}^0$, we have $\fa({\bf \boldsymbol{\lambda}})={\bf \boldsymbol{\lambda}}$, and consequently,
		$$\Logfa(\boldsymbol{\lambda})=(D_{{\bf x}}\fa)({\bf 1})\cdot \Log(\boldsymbol{\lambda}).$$
		By a straight forward computation, we obtain the announced formula. 
	\end{proof}
	
	The map $\Logfa$ can be seen as the representation of $\Loga$ taking the open simplex as a model space for the quotient $\mathbb{R}^{N+1}_{>0}/H_{\bf a}$ and using the identification map $\fa$.

	\textcolor{black}{In order to compute $\Expfa(\boldsymbol{\xi})$, for $\boldsymbol{\xi}\in  T_{\bf{e_a}}\Delta_{N}^0 $, we can proceed as follows: choose a vector $\bold{v}_{\boldsymbol{\xi}}\in \mathbb{R}^{N+1}$ such that $ D_{{\bf x}} \fa (1) \bold{v}_{\boldsymbol{\xi}}=\boldsymbol{\xi}$, then $$\Expfa(\boldsymbol{\xi})=\fa (\Exp(\bold{v}_{\boldsymbol{\xi}})).$$}

	\medskip
Next, we extend the Abelian Lie group $(\simplex,\plusfa)$  to a vector space structure by defining the following scalar multiplication.
	
	\begin{definition} Let $c\in\mathbb{R}$ and $\boldsymbol{\lambda}\in \Delta_N^0$. The scalar multiplication $\orifa$ of $\boldsymbol{\lambda}$ by $c$ is defined as:
		$$c\orifa\boldsymbol{\lambda}:= \Expfa(c \Logfa \boldsymbol{\lambda})$$
	\end{definition}

	\noindent
	An easy consequence of this definition is that : \begin{itemize}
		\item $\Logfa(c\orifa\boldsymbol{\lambda})=c \Logfa(\boldsymbol{\lambda})$, for all $ \boldsymbol{\lambda} \in \simplex$  and $c\in \R$
		\item $c\orifa \Expfa({\boldsymbol{\xi}})=\Expfa(c\boldsymbol{\xi})$, for all $ \boldsymbol{\xi} \in \tanDelta$  and $c\in \R$
	\end{itemize}
	
	\noindent To compute the  scalar multiplication $\orifa$, an explicit expression for $\Expfa$ is not needed. Going back to the definition of $\Logfa$, we have that for any $\boldsymbol{\lambda} \in \simplex$, $ \Logfa(\boldsymbol{\lambda})= D_{\mathbf{x}}\fa(1)\Log (\boldsymbol{\lambda})$.  The commutativity of the diagram, implies that for any ${\bf v} \in T_1\R_{>0}^{N+1}$, $$\Exp_a(D_{\mathbf{x}}\fa(1){\bf v})= \fa(\Exp({\bf v}))$$
	and thus 
	\begin{align*}
	c\orifa\boldsymbol{\lambda}&= \Expfa(c \Logfa \boldsymbol{\lambda})= \Expfa\left(c D_{\mathbf{x}}\fa(1)\Log (\boldsymbol{\lambda})\right)\\
	&=\fa(\Exp(\Log(c\ori\boldsymbol{\lambda})))=\fa((\lambda_1^c,\ldots,\lambda_{N+1}^c)).
	\end{align*}

	\begin{proposition}\label{prop : isometry} The simplex $\Delta_N^0$ endowed with the group composition $\plusfa$ and the scalar multiplication $\orifa$ forms a $\R$-vector space.  Moreover,  $\Expfa$ and $\Logfa$ are linear isomorphisms between $(\tanDelta,+,\cdot)$ and  $(\simplex,\plusfa,\orifa)$.
	\end{proposition}

		Having constructed the logarithm map $\Logfa$ as a linear isomorphism and giving a scalar product $\ps{\cdot,\cdot}$ on the Lie algebra $T_{\bf{e_a}}\Delta_{N}^0$, we can induce a scalar product on $(\simplex,\plusfa,\orifa)$ : 
		\begin{equation*}
		    \psa{ \bl ,\boldsymbol{\mu}}= \ps{\Logfa(\bl),\Logfa(\boldsymbol{\mu})}
		\end{equation*}
		and the associated norm $\|\bl\|_{\fa}=\sqrt{\ps{\Logfa(\bl),\Logfa(\bl)}}$.
		\begin{theorem} \label{thm:dista}
		    We have the following parameterized induced distance functions on the simplex, $\distfa: \Delta_{N}^0\times \Delta_{N}^0\longrightarrow\mathbb{R}$ :
		    \begin{equation*}
		        \distfa(\bl,\boldsymbol{\mu}):=\| \bl \minusfa \boldsymbol{\mu}\|_{\fa}. \footnote{$\bl \minusfa \boldsymbol{\mu} := \bl \plusfa \left((-1)\orifa \boldsymbol{ \mu}\right)$ }
		    \end{equation*}
		    
		    \begin{enumerate}[i)]
		        \item These distances are bi-invariant, \textit{i.e.} $\distfa(\boldsymbol{\gamma}\plusfa\boldsymbol{\lambda},\boldsymbol{\gamma}\plusfa\boldsymbol{\mu})=\distfa(\boldsymbol{\lambda},\boldsymbol{\mu})$, for all $\boldsymbol{\gamma,\lambda,\mu} \in \simplex$. 
		        \item The maps $\Logfa$ and $\Expfa$ are linear isometries between   $(\simplex,<\cdot,\cdot>_{\fa})$  and  $(\tanDelta, <\cdot,\cdot>)$.
		        \item Choosing the Euclidean scalar product on $\tanDelta$, we obtain the following formula :
		        	\begin{align*}
	\distfa(\boldsymbol{\lambda},\boldsymbol{\mu})^{2}&=\sum_{i=1}^{N+1}({\bf e}_{\bf a})^2_{i}\left( \ln(\frac{\lambda_i}{\mu_i})-\frac{a_i}{\sum_{k=1}^{N+1}a_k({\bf e_{\bf a}})_k}\sum_{j=1}^{N+1}({\bf e}_{\bf a})_{j}\ln(\frac{\lambda_j}{\mu_j})\right)^2.
	\end{align*}
		    \end{enumerate}
		       
		\end{theorem}
The Euclidean scalar product is a  natural choice on $\tanDelta$ because this is an Abelian Lie algebra. Another choice would be the  Killing scalar product but this is zero on an Abelian algebra. Consequently, in what follows, we consider the Euclidean scalar product on $\tanDelta$.	
		
		\medskip

	\medskip
	
	Putting together the previous diagrams, we obtain the following depiction of the constructions we have made so far: 
	
	$$\xymatrix{
		(\mathbb{R}^{N+1}_{>0},\plus)  \ar@<1ex>[rrr]^*{\Log} \ar[d]^*{\pi_{\bf a}} \ar[ddr]|\hole|\hole^>>>>>>*{\fa}	&&& T_{\bf 1}\mathbb{R}^{N+1}_{>0}\ar[lll]^*{\Exp} \ar[d]_*{D_{\bf x}\pi_{\bf a}({\bf 1)}} \ar[ddrr]^*{D_{\bf x}\fa({\bf 1})} & & \\
		(\mathbb{R}^{N+1}_{>0}/H_{\bf a},\diamond_{\bf a}) \ar@<0.5ex>[rrr]^*{\Log_{\bf a}} \ar[dr]_*{s_{\bf a}} &&& T_{[{\bf 1}]}\left( \mathbb{R}^{N+1}_{>0}/H_{{\bf a}}\right) \ar[lll]^*{\Exp_{\bf a}} \ar[drr]_*{\tilde{s}_{\bf a}} &  &\\
		&(\Delta_{N}^0, \plusfa)\ar@<1ex>[rrrr]^*{\Logfa}  && & &  T_{\bf{e_a}}\Delta_{N}^0 \ar[llll]^*{\Expfa}
	}$$
	
	Even though $\Delta_N^0$ is a subset of  $\R^{N+1}_{>0}$, it is not a subgroup of $(\R^{N+1}_{>0},\plus)$ and thus the quotient operation $\plusfa$ is not simply the restriction of $\plus$ to $\Delta_N^0$. This further implies that $\Delta_N^0$ is non-linear (not a linear subspace) in the ambient vector space $(\R^{N+1}_{>0},\plus,\ori)$.
	
	\textcolor{black}{As we have seen so far, the quotient group $\mathbb{R}^{N+1}_{>0}/H_{\bf a}$ exists for any choice of ${\bf a}\in \mathbb{R}^{N+1}$ and we can regard ${\bf a}$ as a parameter for the quotient structures. Nevertheless, only the choices with $a_i>0$ (or $a_i<0$), for all $i\in \overline{1,N+1}$ can accommodate the simplex  $\Delta_N^0$ as a model space. The simple observation that $H_{\bf a}=H_{c\bf a}$, for any scalar $c\in \mathbb{R}$ shows that the quotient structure depends only on the line $[{\bf a}]_{_{\bf 1}}$ and consequently, the quotient structures are parametrized by the real projective space ${\bf RP}^N$. The simplex  $\Delta_N^0$ can be a model space for those quotient structures parameterized by $[{\bf a}]_{_{\bf 1}}\in {\bf RP}^N$ with $a_i>0$ (or $a_i<0$), for all $i\in \overline{1,N+1}$. Everything can be organized in a bundle of groups (Lie groupoid) (see, \cite{groupbundle}) with base ${\bf RP}^N$ but this is beyond the purpose of this paper.}
	
	\textcolor{black}{A difficulty that we encounter is the computation of the neutral element ${\bf{e_a}}\in \Delta_N^0$ for a given quotient structure parametrized by $[{\bf a}]_{_{\bf 1}}$. Nevertheless, we can solve the converse problem noticing that for any $\boldsymbol{\lambda}\in\Delta_N^0$, the class $[(\cdots,ln \lambda_i,\cdots)]_{_{\bf 1}}$ generates the quotient Lie group structure where $\boldsymbol{\lambda}$ is the neutral element. We have the following result.}
	
	\textcolor{black}{
		\begin{lemma}
			The relation between the parameter $[{\bf a}]_{_{\bf 1}}$ with $a_i>0$ (or $a_i<0$), for all $i\in \overline{1,N+1}$ and the neutral element of the quotient Lie group structure it generates on $\Delta_N^0$ is given by the equality
			$$[{\bf a}]_{_{\bf 1}}=[\cdots,\ln\left(({\bf e_a})_i\right),\cdots]_{_{\bf 1}}.$$
	\end{lemma}}
	
	\bigskip

	\section{Particular cases}\label{sec : 3}
		We illustrate the above constructions for some basic choices of $\bf{a}$ that allow for explicit computations. We start with the basic case ${\bf{a}}=(1,\ldots,1)$, that gives a new perspective from a Lie group theory point of view on several objects encountered in the framework of compositional data.
	\subsection{Case ${\bf{a}}=(1,\ldots,1)$ \textit{i.e.}, Compositional Data analysis}\label{sec: composition}
	Consider the case $\bf{a}={\bf 1}$, \textit{i.e.}  $a_i=1, i=\overline{1,N+1}$. Then 
	\begin{equation*}\label{defH1}
	H_{\bf a}=H_{\bf 1}=\left\{\Exp(t{\bf 1})=(\dots,e^{t},\dots)|~  t\in \mathbb{R}\right\}.
	\end{equation*}
	 In one of his seminal works \cite{conciseguide}, Aitchison says that “When we say that a problem is compositional we are recognizing that the sizes of our specimens are irrelevant.” This is known as the scale invariance principle. When dealing with classical Compositional Data analysis, this comes down to saying that two specimen vectors  ${\bf{v}},{\bf{w}}\in \R_{>0}^{N+1}$ are equivalent from a compositional point of view iff $\bf{v}=\alpha \bf{w},$ with $\alpha\in \R_{>0}^+$. \textcolor{black}{The above equivalence can be rewritten in terms of proportions as $\dfrac{v_i}{v_{j}}=\dfrac{w_i}{w_{j}}$, for all $i,j\in \overline{1,N+1}$.} This, of course, corresponds to considering \textcolor{black}{equivalence classes generated by $H_{\bf 1}$ and consequently,} the quotient with respect to $H_{\bf 1}$.\\
	 
	In this case it is easy to see that for all $\textbf{x} \in \mathbb{R}^{N+1}_{>0} $, $$t_{\bf{1}}(\textbf{x})=-\ln(\|{\bf x}\|)\quad
	\mbox{ and } \quad \fu(\textbf{x})=\frac{\textbf{x}}{\|{\bf x}\|},$$
	where $\|\cdot\|$ denotes the $L^1$ norm and $t_{\bf{1}}({\bf x})$ is the solution of equation \eqref{implicit eq}.  
The map	$\fu$ is known in the Compositional Data analysis literature as the closure operation.

	The corresponding quotient group structure can be easily deduced: for all $\boldsymbol{\lambda},\boldsymbol{\mu} \in \Delta_N^0$,
	$$\boldsymbol{\lambda}\plusunu \boldsymbol{\mu}= \dfrac{\boldsymbol{\lambda} \plus \boldsymbol{\mu}}{\|\boldsymbol{\lambda} \plus \boldsymbol{\mu}\|} = \fu(\boldsymbol{\lambda}\plus \boldsymbol{\mu}).$$
	This is what Aitchison has defined as the perturbation operation in \cite{Aitchison1982TheSA}.
	
	Furthermore, the neutral element $\bf{e_1}$ is given by $\left(\ldots,\dfrac{1}{N+1},\ldots\right)$.
	Thus, the derivative of $\fu$ given by $\eqref{formula Dx}$ becomes:
	\begin{equation*}
	(D_{{\bf x}}\fu)_{i,j}({\bf 1})=\delta_{i,j}\frac{1}{N+1}-\frac{1}{(N+1)^2}.
	\end{equation*}
	
	From Theorem \ref{thm: log} it follows that, for all $\boldsymbol{\lambda} \in \Delta_{N}^0$:
	
	\begin{align*}
	\Logfunu(\boldsymbol{\lambda})&= \left(\ldots, \frac{\ln(\lambda_i)}{N+1}-\frac{1}{(N+1)^2}\sum_{j=1}^{N+1}\ln \lambda_j , \ldots\right)\\
	&= \frac{1}{N+1} \Log (\boldsymbol{\lambda})+\frac{1}{(N+1)^2}\| \Log (\boldsymbol{\lambda})\| {\bf 1}
	\end{align*}
	In particular, this formula implies that, up to a constant, the well-known centered log-ratio transformation (see for example: \cite{Aitch86}, \cite{GreenacreBook}) can actually be seen as the logarithmic map $\Logfunu$ of the quotient Lie group  $(\simplex,\plusunu)$.  \\
	The associated distance $\distfunu: \Delta_{N}^0\times \Delta_{N}^0\longrightarrow\mathbb{R}_+,$ can be written as:
	$$\distfunu(\boldsymbol{\lambda},\boldsymbol{\mu})^2=\frac{1}{(N+1)^2}\sum_{i=1}^{N+1}\left( \ln \frac{\lambda_i}{\mu_i}- \frac{1}{N+1} \sum_{j=1}^{N+1}\ln \frac{\lambda_j}{\mu_j} \right)^2$$

This distance, sometimes called Aitchison distance, first appeared in \cite{Aitchinson2001} and \cite{Aitch2002}. \\

	For this particular case of ${\bf a}$, we can give an explicit formula for $\Expfunu$. This is facilitated by the property that for all $\boldsymbol \xi \in T_{\bf{e_1}}\Delta_{N}^0 =\{ \boldsymbol{ \eta} \in \mathbb{R}^{N+1} \  | \displaystyle\sum_{i=1}^{N+1} \eta_i=0 \}$, $D\fu({\bf 1}) \boldsymbol{\xi}=\frac{1}{N+1}\boldsymbol\xi $ and thus, the commutativity of the diagram \eqref{diagrama fa} implies that :
	
	\begin{align*}
	\Expfunu(\boldsymbol\xi)&= \Expfunu((N+1)D\fu({\bf 1}) \boldsymbol\xi)\\
	&= \fu\left(\Exp ((N+1)\boldsymbol{\xi}) \right)=\\
	&= \dfrac{\Exp((N+1)\boldsymbol{\xi})}{\| \Exp((N+1)\boldsymbol{ \xi})\|}, \quad \mbox{ for all } \boldsymbol\xi \in  T_{\bf{e_1}}\Delta_{N}^0.
	\end{align*}
	This exponential map is also known as the softmax function encountered in logistic regression and neural networks. 

  To summarise, we give a dictionary between the constructions resulting from the Lie group theory and the ones encountered in compositional data analysis. 
	
	\begin{table}[h]
		\centering
		\begin{tabular}{|c|c|}\hline
			\textbf{  Compositional Data}  &  \textbf{Lie Group Theory}  \\ \hline
			closure  operation   &  $\fu:\R^{N+1}_{>0}\to \Delta_N^0$\\ \hline
			perturbation & $\plusunu :\Delta_N^0\times \Delta_N^0 \to \Delta_N^0$\\ \hline
			power operation & $\oriunu :\mathbb{R}\times \Delta_N^0 \to \Delta_N^0$\\ \hline
			Centered Log Ratio\tablefootnote{This correspondence is up to the constant factor (N+1)} (crl) & $\Logfunu : \Delta_N^0\to  T_{\bf{e_1}}\Delta_{N}^0$ \\ \hline
			crl$^{-1}$ & $\Expfunu:  T_{\bf{e_1}}\Delta_{N}^0\to \Delta_{N}^0$\\ \hline
			Aitchison distance & $\distfunu:\simplex\times\simplex\to R_{+}$\\ \hline 
			
		\end{tabular}
		\label{tab:my_label}
	\end{table}

	\subsection{Case ${\bf a}=(1,\ldots,1,2)$}
	
	To highlight the impact of the reference subgroup $H_{\bf a}$ on the quotient group structure and the afferent objects, let us consider a second example. Set ${\bf a}=(1,\ldots,1,2)$ and consider the associate subgroup:
	$$H_{\bf a}=\left\{{\bf exp}(t{\bf a})=(e^{t},\dots, e^t,e^{2t})|~ t\in \mathbb{R}\right\}.$$  Applying $\eqref{def fa}$ implies, for all ${\bf{x}}\in \mathbb{R}^{N+1}_{>0}$, that the intersection of the equivalence class of ${\bf x}$ with the simplex is given by :
	$\fa({\bf x})=h_{\bf x, \bf a}(t_{\bf a}(\bf x)),$
	where  $t_{\bf a}(\bf x)$ is the solution of the equation
	\begin{equation}\label{eq ta}
	\left(\sum_{i=1}^{N} x_{i}\right)e^{t}+x_{N+1}e^{2t}=1.
	\end{equation}
	Thus $\fa$ can be explicitly computed. Denoting $S_{N}= \sum_{i=1}^{N} x_{i}$, \eqref{eq ta} can be written as: 
	\begin{equation} \label{eq ta y}
	x_{N+1}y^2+S_Ny-1=0, \quad \mbox{ with } y =e^{t}
	\end{equation}
	Hence a solution of $\eqref{eq ta}$ can be obtained from the positive solution of \eqref{eq ta y},
	$$y_{+}=\dfrac{-S_N+\sqrt{S_N^2+4x_{N+1}}}{2x_{N+1}},$$
	implying that $t_{\bf a}({\bf x})= \ln(y_+)$ and consequently :
	$$\fa({\bf x})=\left(\ldots,x_i \dfrac{-S_N+\sqrt{S_N^2+4x_{N+1}}}{2x_{N+1}},\ldots, x_{N+1}\left( \dfrac{-S_N+\sqrt{S_N^2+4x_{N+1}}}{2x_{N+1}}\right)^2 \right).$$
	In this case, the neutral element  can be explicitly computed and we have that : $${\bf e_{ a}}=\fa(\textbf{1})=\left(\frac{-N+\sqrt{N^2+4}}{2},\ldots,\left(\frac{-N+\sqrt{N^2+4}}{2} \right)^2\right).$$
	For all $\boldsymbol{\lambda} \in \Delta_N^0$,  
	$$\left(\Logfa \boldsymbol\lambda\right)_i= \frac{-N+\sqrt{N^2+4}}{2}\ln \lambda_i+\frac{1}{\sqrt{N^2+4}}\| \Log \boldsymbol{\lambda}\|_{\bf e_a}, \quad 1\le i\le N$$
	and $$\left(\Logfa \boldsymbol\lambda\right)_{N+1}=\frac{-N+\sqrt{N^2+4}}{2}\left(\frac{-N+\sqrt{N^2+4}}{2}\ln \lambda_{N+1}+\frac{2}{\sqrt{N^2+4}}\| \Log \boldsymbol \lambda\|_{\bf e_a} \right),$$
	where $\| \cdot\|_{\bf e_a} $ denotes the $L^1$ norm weighted by ${\bf e_a}$.

	Finally, the associated distance is given by :
	\begin{align*}
	\distfa(\boldsymbol{\lambda,\mu})   &=\sum_{i=1}^N  \left(\frac{-N+\sqrt{N^2+4}}{2}\right)^2\left( \ln\frac{\lambda_i}{\mu_i}-\frac{M_N}{\sqrt{N^2+4}} \right)\\
	& +\left(\frac{-N+\sqrt{N^2+4}}{2} \right)^4\left( \ln\frac{\lambda_{N+1}}{\mu_{N+1}}-\frac{2M_N}{\sqrt{N^2+4}}\right)
	\end{align*}
	where $ M_{N}:=\sum_{j=1}^{N}\ln\frac{\lambda_j}{\mu_j}+\frac{-N+\sqrt{N^2+4}}{2}\ln\frac{\lambda_{N+1}}{\mu_{N+1}} $
\section{Generalised compositional data }	\label{sec:4}
As we have seen before, the main concepts and mathematical objects used in compositional data arise naturally in the Lie group theory, when considering the simplex as a space model for the quotient group $\R^{N+1}_{>0}/H_{\bf 1}$. If the nature of the data implies a different link between the components, considering a different equivalence relation,  determined by a general $H_{\bf a}$,  might be a more appropriate choice. There are several principles, introduced by Aitchison, that are considered being fundamental in an large part of the compositional data literature. A detailed historical account of these principles and an analysis of their usefulness when studying data is given in \cite{cocktails}. In what follows, we explain what these principles become in the more general framework when the equivalence relation is generated by $H_{\bf a}$.

\paragraph{Scale invariance} The scale invariance principle used in compositional data corresponds to the equivalence relation determined by $H_{\bf 1}$. We give a straightforward generalisation of this principle when the equivalence class is determined by $H_{\bf a}$. In this case,  two specimens $\bf{v}$ and $\bf{w}$ belong to the same class determined by $H_{\bf a}$, \textit{i.e.}
		$\bf{v} \sim \bf{w} $, iff there exists $\alpha \in \mathbb{R}_{>0}$ such that ${\bf{v}}=(w_1\alpha^{a_1},\ldots,w_{N+1}\alpha^{a_{N+1}})$ \textcolor{black}{or equivalently, $\dfrac{{v_i}^{a_j}}{v_{j}^{a_i}}=\dfrac{{w_i}^{a_j}}{w_{j}^{a_i}}$, for all $i,j\in \overline{1,N+1}$. \\
		If the two specimens are two different instances of a variable $X:\mathbb{R}\rightarrow \R_{>0}^{N+1}$, i.e. ${\bf v}=X(\tau_1)$ and  ${\bf w}=X(\tau_2)$, where $\tau$ is some parameter describing mass, or volume, time, concentration, etc., then the equivalence relation above can be restated as $\dfrac{\delta}{\delta \tau}X(\tau):=X(\tau)^{-1}\dot{X}(\tau)=\bf{a}$. The operation $\dfrac{\delta}{\delta \tau}$ is the logarithmic derivative in the Abelian group $(\R_{>0}^{N+1},\plus)$.}  This leads to considering the quotient structure with respect to $H_{\bf a}$.
	Such situations are encountered in practice, for example, when analyzing the mass evolution of components in a system, see \cite{usesandmisuses}. For instances, when the data comes from an exponential decay dynamic, a situation also described in \cite{usesandmisuses}, the Aitchison geometry (case $\bf{a}=\bf{1}$) might not be the best approach when dealing with general equivalence classes.\\
		Another case when a more general geometry is better adaptive is for data in which some components are measured in $unit$ and other, for example, in $unit^2$ or $unit^3$ (for instance one component that represents a length and another an area or a volume).

			\textcolor{black}{\paragraph{Permutation invariance} This principle refers to the fact that the results of an analysis of the data should be invariant under any permutation of the components, as long as the same permutation is applied to the whole data set. This property still holds for the generalized framework corresponding to any ${\bf a }\in \R^{N+1}_{>0}$, as soon as the quotient structure is adapted accordingly : if a permutation $\sigma$ is used on the components of the data, then the proper quotient structure is the one with respect to $H_{ \sigma({\bf a })}$.}
	
	\paragraph{Subcompositional coherence}

	In what follows, we define a subcomposition procedure in the previously proposed general framework and  show that the subcomposition procedure is a linear map between two vector spaces.
	
	$$ \xymatrix{
		\left(\mathbb{R}^{N+1}_{>0},\oplus,\ori\right)  \ar@<1ex>[rr]^*{\alg} \ar[d]^*{\fa}	&& \left(\mathbb{R}^{k+1}_{>0}, \oplus,\ori \right) \ar[d]^*{\fap} \\
		\left(\Delta_N^0,\plusfa,\orifa\right) \ar@<1ex>[rr]^*{\sub} &&\left(\Delta_k^0,\oplus_{\fap},\ori_{\fap} \right)  }$$
	
	We denote by $\alg:\mathbb{R}^{N+1}_{>0}\to\mathbb{R}^{k+1}_{>0}$ the map that assigns to each ${\bf x}=(x_1,\ldots,x_{N+1})\in \mathbb{R}^{N+1}_{>0} $ a sub-part of $k+1$ elements $(x_{\sigma_1},\ldots,x_{\sigma_{k+1}})$ (not necessarily in the initial order) and set ${\bf a}':=\alg({\bf a})$. Then, for $\boldsymbol{\lambda}\in \Delta_{N}^0$ we can define the subcomposition corresponding to the $k+1$ elements chosen by $\alg$ as : 
	\begin{equation*}
	\sub(\boldsymbol{\lambda}):=\fap(\alg(\boldsymbol{\lambda})).
	\end{equation*}
	The well definiteness of the operator is a consequence of the fact that $\forall {\bf x}\in  \mathbb{R}^{N+1}_{>0} $, $\fap(\alg(\bx))=\fap(\alg(\fa(\bx)))$.
	Indeed, let $\bx\in \R^{N+1}_{>0}$. Then $\fa(\bx)=(x_1e^{\tax a_1},\ldots,x_{N+1}e^{\tax a_{N+1}} )$, where $\tax$ is the solution of the equation \eqref{implicit eq} and 
	\begin{align*}
	\alg(\fa(\bx)) &=(x_{\sigma_1}e^{\tax a_{\sigma_{1}}},\ldots,x_{\sigma_{N+1}}e^{\tax a_{\sigma_{N+1}}}) \\
	&=(x_{\sigma_1}e^{\tax a_1'},\ldots,x_{\sigma_{N+1}}e^{\tax a_{N+1}'}) \in \alg(\bx)H_{\bf a'}.
	\end{align*}
	In other words, we have that $\alg(\fa(\bx))$ is in the same equivalence class as $\alg(\bx)$ with respect to $H_{\bf a'}$ and thus by definition $\fap(\alg(\fa(\bx)))=\fap(\alg(\bx))$.

	The fact that $\sub$ is a linear map is now an immediate consequence of this last equality and the fact that $\fap$ and $\alg$ are themselves linear maps.

 Subcompositional coherence principle is generally described as the fact that the results obtained for a sub-part of a compositional data (studied after closure) should remain coherent with the results obtained for the full composition. A classical example of subcompositional incoherence, shows that when compositional data framework is not taken into account, we can obtain a positive correlation between two variables when they are studied as composites of a specimen with $n$ components and a negative correlation when the same variables are studied on a (reclosed) subcomposition of the same original specimen (see for example \cite{Aitchinson2001}). Of course, not all results obtained on a compositional data are preserved by the $\sub$ operator. For example, the arithmetic mean on the simplex is subcompositional coherent in the sense that : $\sub\left(\frac{1}{n}\orifa\left( {\bl_1 \plusfa\ldots\plusfa\bl_n}\right)\right)= \frac{1}{n}\ori_{\bf a'}\left(\sub\left({\bl_1}) \plus_{\fap}\ldots\plus_{\fap}\sub({\bl_n}\right)\right).$ 
	\section{Elements of descriptive statistics on the simplex and normal distribution}\label{sec:5}
	\paragraph{Centrality for generalized compositional data}
	\textcolor{black}{	Once we endow the simplex with a distance $\distfa$, a natural first notion of centrality that can be considered is the corresponding Fréchet mean. Fréchet means were introduced by M Fréchet in   \cite{frechet} and represent a generalization of the classical expected value to arbitrary metric spaces.
 For a random variable $X$, defined on a metric space $(\mathcal{M},d)$, a Fréchet mean is a minimizer of $ y \mapsto \mathbb{E}\left[ d^2(y,X)\right]$. When there exists a unique minimizer $m(X)$, we say that $m(X)$ is the Fréchet mean of $X$ with respect to $d$.
	  Aitchison used this notion to define the geometric mean as a measure of central tendency (see for example \cite{conciseguide}) for compositional data.
	A similar result can be obtained in a straightforward way for a more general distance $\distfa$. Let $X$ be a random variable on $\simplex$. Then a Fréchet mean of $X$ with respect to $\distfa$ is a minimizer of the real valued functional:  
	\begin{equation}\label{FrechetMean}
	\boldsymbol{\lambda} \mapsto \mathbb{E}\left[ \distfa^2(\boldsymbol{\lambda},X)\right].
	\end{equation}
	Now, for any $ \boldsymbol{\lambda,\mu}\in\simplex$, the definition of the distance $\distfa$ implies that $\distfa(\boldsymbol{\lambda,\mu})=d(\Logfa(\boldsymbol{\lambda}),\Logfa(\boldsymbol{\mu}))$. We also have that, for all $\boldsymbol{\lambda}\in \simplex$, $\mathbb{E}[\distfa^2(\boldsymbol{\lambda},X)]=\mathbb{E}[d^2(\Logfa(\boldsymbol{\lambda}),\Logfa(X)]$, thus a minimizer of \eqref{FrechetMean} is a Fréchet mean of $\Log_a(X)$ and therefore, as soon as $\Logfa(X)$ is integrable, we have that $m_a(X)$, the Fréchet mean of $X$ with respect to $\distfa$,  is unique and that :
	\begin{equation*}
	m_{\bf a}(X)=\Expfa\left(\mathbb{E}\left[\Logfa(X)\right]\right).
	\end{equation*}
	This last result can also be seen as a consequence of Proposition 1 of \cite{GeometricAproachPawlosky-Egozcue2001}  since  $\Logfa$ is an isometry.}
	
		Since the simplex is endowed with additional mathematical structure we can also construct the expected value of a random variable. Let ${\bf a }\in \mathbb{R}_{>0}^{N+1}$ and denote $\mathcal{B}_{\bf a}=\Expfa(\mathcal{B})$, where $\mathcal{B}$ is the Borel $\sigma$-algebra of $T_{e_{\bf a}}\Delta_{N}^0$. $\mathcal{B}_{\bf a}$ is also the smallest $\sigma$-algebra containing all the open balls associated to the distance $\distfa$. Following the same procedure as \cite{GeometricAproachPawlosky-Egozcue2001}, for each random variable $X :(\Omega,\mathcal{F},\mathbb{P})\to (\Delta_{N}^0, \mathcal{B}_{\bf a})$, we can consider the real valued random vector  $\Logfa(X) :(\Omega,\mathcal{F},\mathbb{P})\to (T_{e_{\bf a}}\Delta_{N}^0, \mathcal{B})$ and obtain that :
	
\begin{equation*}
\mathbb{E}_{\bf a }[X]=\Expfa\left( \mathbb{E}[\Logfa(X)]\right).
\end{equation*}	
	It is well known that in $\R^N$, any integrable random variable $Z$ (\textit{ie} for which $\mathbb{E}[Z]$ exists and is well defined) has a unique Fréchet mean with respect to the euclidean distance $d$ and that $m(Z)=\mathbb{E}[Z]$. This property holds here too, in the sense that when $\mathbb{E}_{\bf a }[X]$ is well defined,  $m_{\bf a }(X)=\mathbb{E}_{\bf a }[X]$.
	\paragraph{PCA}
	
	\textcolor{black}{Proposition \ref{prop : isometry} implies in particular that $\Logfa$ can be seen as the equivalent of crl map for generalized compositional data.
		With this is mind, all the statistical methods based on log-ratios used for compositional data can be extended in a straightforward way to generalized compositional data. Take for example principal component analysis (PCA). This is a well known multivariate statistical analysis method, often used for dimensionality reduction. The main idea of PCA is to synthesize a data set consisting in a large number of correlated variables, while retaining as much as possible the variation present in
		it. To achieve this, PCA computes new uncorrelated variables, called principal components. For more details see \cite{joliffe}.}
		
		\textcolor{black}{In an Euclidean framework, PCA can be reduced to the study of eigenvalues and eigenvectors of a positive-semidefinite symmetric matrix, and is generally treated in this manner.
		In the context of generalized compositional data, principal component analysis can be applied on the matrix obtained through a $\Logfa$ transformation of the data set. This is the equivalent of logratio analysis used for classical compositional data.}
		
	\textcolor{black}{	To highlight the link between the PCA on the image of a data set through $\Logfa$ and an analysis directly on the simplex, it is convenient to look at the variational formulation of the PCA.}
		
\textcolor{black}{Indeed, principal components can be described as a sequence of nested affine sub-spaces of increasing dimension that maximize the variance of the projections or minimize the sum of norms of projection residuals. For example, consider the case of $m$ points ${\bf x}_1,\ldots {\bf x}_m$, in $\mathbb{R}^n$ and for all ${\bf v}\in\mathbb{R}^n$ denote: 
		\begin{equation*}
		S_{\bf v}=  \{{\bf \bar{x}}_m+t{\bf v}, t\in \mathbb{R}\}, \mbox{ where  } {\bf \bar{x}}_m=\frac{1}{m}\sum_{i=1}^m {\bf x}_i 
		\end{equation*}  
		A first principal component of the data can be described as $S_{{\bf v}_1}$, where ${\bf v}_1$ is such that:
		\begin{equation*}
		{\bf v}_1 \in  \argmin\limits_{\|{\bf v}\|=1} \sum_{i=1}^m d^2({\bf x}_i,S_{\bf v}).
		\end{equation*}
	This formulation is at the root of several generalizations of PCA to non-euclidean framework (like  on the  Wasserstein space of probability measures over $\mathbb{R}$ in \cite{Bigot}, for a survey of methods employed in Riemannian settings see \cite{axel}). }

Now consider the equivalent objects on $(\Delta_{N}^0,\plusfa,\orifa)$. Suppose we have $\bl_{1},\ldots\bl_{m}\in \Delta_{N}^0$ and for all $\boldsymbol{\mu}\in \Delta_{N}^0$ denote :
\begin{equation*}
S_{\boldsymbol{\mu}}=\{ \bar{\bl}_m\plusfa t\orifa \boldsymbol{\mu}, \ t \in \mathbb{R} \}, \mbox{ with } \ \bar{\bl}_m=\frac{1}{m}\orifa\left[\bl_1 \plusfa\ldots\plusfa \bl_m\right]
\end{equation*}
A first principal component becomes $S_{\boldsymbol{\mu}_1}$, with $\boldsymbol{\mu}_1$ such that :
	\begin{equation*}
{\boldsymbol{\mu}_1} \in  \argmin\limits_{\|{\boldsymbol{\mu}}\|_{_{\fa}}=1} \sum_{i=1}^m d_{\fa}^2({\boldsymbol{\mu}}_i,S_{\boldsymbol{\mu}}).
\end{equation*}
For $i\in\{1,\ldots m\}$, set ${\bf x}_i= \Logfa(\bl_i)$. Let $\boldsymbol{\mu} \in \Delta_N^0$ and denote ${\bf v }=\Logfa(\boldsymbol{\mu})$. Using the previous notations and the definition of the distance given by Theorem \ref{thm:dista} we have that :
 \begin{equation*}
 d_{\fa}^2({\boldsymbol{\mu}}_i,S_{\boldsymbol{\mu}}) = d^2({\bf x}_i,S_{\bf v}).
 \end{equation*}
Keeping in mind that $\|{\boldsymbol{\mu}}\|_{_{\fa}}=\| \Logfa(\boldsymbol{ \mu})\|=\|{\bf{v}}\|$, we have that 
\begin{equation*}
{\boldsymbol{\mu}_1} \in  \argmin\limits_{\|{\boldsymbol{\mu}}\|_{_{\fa}}=1} \sum_{i=1}^m d_{\fa}^2({\boldsymbol{\mu}}_i,S_{\boldsymbol{\mu}}) \Longleftrightarrow \Logfa(\boldsymbol{\mu}_1) \in \argmin\limits_{\|{\bf v }\|=1} \sum_{i=1}^m d^2({\bf x}_i,S_{\bf v })
\end{equation*}

And thus,  $S_{\boldsymbol{ \mu}}$ is a first principal component for $\bl_{1},\ldots\bl_{m}\in \Delta_{N}^0$ if and only if $S_{\Logfa(\boldsymbol{ \mu})} $ is a principal component for $\Logfa( \bl_{1}),\ldots \Logfa(\bl_{m})$.

This argument can be easily extended to principal components of higher order, showing thus that applying  PCA on the data transformed through $\Logfa$ is the same as using it directly on the simplex with the appropriate operations.

	\paragraph{Normal distribution on the simplex}
	The vector structure also facilitates the definition of a Gaussian vector on the simplex.
	We can apply for example the definition of a Gaussian vector on a general vector space given by \cite{MorrisEaton} 
	to define a  normal law on the simplex with respect with the structure corresponding to $H_{\bf a }$.
	\begin{definition}\label{def:normal}
		We say that a random vector $X :(\Omega,\mathcal{F},\mathbb{P})\to (\Delta_{N}^0, \mathcal{B}_{\bf a},\plusfa,\orifa)$ has a normal distribution on the simplex  iff, for all $\bl\in\simplex$,  $\psa{\bl,X}=$ $ \langle \Logfa(\bl),\Logfa(X)\rangle$  has a normal distribution on $\mathbb{R}$, \textit{i.e.} $\psa{\bl,X}\sim \mathcal{N}(\mu_{\bl},\sigma_{\bl}^2)$ for some  $\mu_{\bl},\sigma_{\bl}\in \mathbb{R}$.
	\end{definition}
Let $B=({\bf b}_1,\ldots,{\bf b}_N)$ be an orthonormal basis of $(\simplex,\plusfa,\orifa)$. Then $X$ is a normal random vector iff there exists a real multivariate normally distributed vector $Z=(Z_1,\ldots,Z_N)$ such that

 \begin{equation}\label{eq : X normal comb}
 X=Z_1\orifa{\bf b}_1\plusfa\ldots\plusfa Z_N\orifa{\bf b}_N.
 \end{equation}

To see this, first observe that $X=\sum_{i=1}^{N}\psa{{\bf b}_i,X}\orifa {\bf b}_i$\footnote{Here the sum is considered with respect to operation $\plusfa$}. Thus if Definition \ref{def:normal} holds then there exist $\mu_i,\sigma_i\in \mathbb{R}$ such that $\psa{{\bf b}_i,X}\sim\mathcal{N}(\mu_i,\sigma_i^2)$. Denote $Z_i=\psa{{\bf b}_i,X}$ and $Z=(Z_1,\ldots Z_N)$. 
Let ${\bf c}\in \mathbb{R}^{N}$ and set $\bl=c_1\orifa {\bf b}_1\plusfa\ldots\plusfa c_N\orifa {\bf b}_N \in \simplex$. Since $B$ is an orthonormal basis, we have that  $\psa{X,\bl} = \sum_{i=1}^{N}c_iZ_i$.  
  According to Definition \ref{def:normal}, this last term is Gaussian. Therefore we can conclude that any linear combination of components of $Z$ is Gaussian and thus $Z$ is a multivariate normal vector. 

Conversely if equality \eqref{eq : X normal comb} holds, then for all $\bl \in \simplex,$ 

\begin{align*}
\psa{\bl,X} &= \ps{\Logfa(\bl),\Logfa(X)}= \ps{\Logfa(\bl),\sum_{i=1}^{N}Z_i\Logfa({\bf b}_i)}\\
&=\sum_{i=1}^{N}Z_i\ps{\Logfa(\bl),\Logfa({\bf b}_i)}.
\end{align*}
This shows that $\psa{\bl,X} $ is a linear combination of the components of a real normal vector and thus has a normal distribution. \\
When $X$ is defined by $\eqref{eq : X normal comb}$, its expected value can be written as:
\begin{align*}
\mathbb{E}_{\bf a}[X]&=\mathbb{E}_{\bf a}\left[ Z_1\orifa{\bf b}_1\plusfa\ldots\plusfa Z_N\orifa{\bf b}_N \right] \\
&=\mathbb{E}[Z_1]\orifa{\bf b}_1\plusfa\ldots\plusfa \mathbb{E}[ Z_N]\orifa{\bf b}_N, \\
\end{align*}
and so the coordinates of $\mathbb{E}_{\bf a}[X]$ in the basis $B$ are given by $\mathbb{E}[Z]$. We can also define the covariance matrix of $X$ with respect to the basis $B$ as $(\Sigma_B)_{ij}=\mathrm{Cov}(Z_i,Z_j)$.

Considering the special case, when $Z_1,\ldots Z_N$ are $\mathcal{N}(0,1)$ i.i.d. real random variables, 
we have that:

\begin{align*}
\mathbb{E}_{\bf a}[X]&=\mathbb{E}_{\bf a}\left[ Z_1\orifa{\bf b}_1\plusfa\ldots\plusfa Z_N\orifa{\bf b}_N \right] ={\bf e_a}
\end{align*}

and the covariance matrix of $X$ with respect to the basis $B$ is $I_{N}$.

\begin{figure}[htb!]
	\hfill
	\subfigure[Case $a=(1,1,1)$]{\includegraphics[width=5cm]{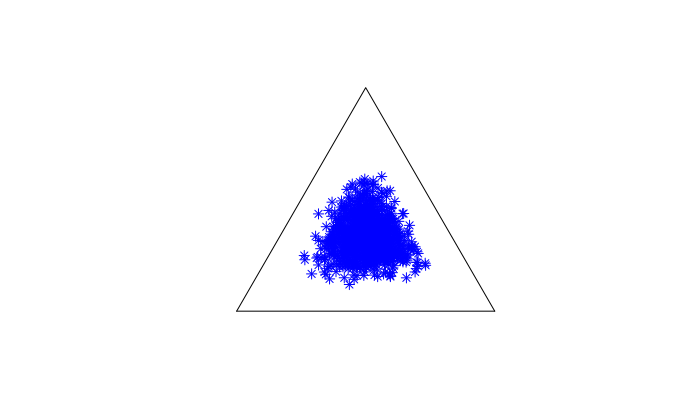}}
	\hfill
	\subfigure[Case $a=(1,1,2)$]{\includegraphics[width=5cm]{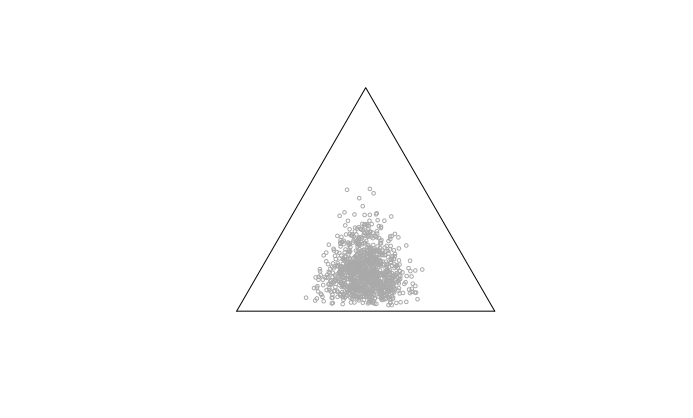}}
	\hfill
	\caption{Simulation of $1000$ realizations of centered normal random variables on the simplex endowed with two different structures }
\end{figure}

Let $\mu_B,\Sigma_B$ be the parameters of a normal random vector $X$ with respect to a fixed orthonormal basis $B$ and suppose that $\Sigma_B$ is non-degenerate (of rank $N$). Denote $T:\simplex\to \mathbb{R}^N$ the bijective linear map that associates to each $\bl \in \simplex$ its coordinates in the basis $B$. Define the measure $\nu_T$ on the Borel sets $A \in \mathcal{B}_{\bf a}$, by $\nu_T(A)=\nu (T(A))$, where $\nu$ denotes the Lebesgue measure on $\mathbb{R}^N$. Then the density function of $X$ with respect to $\nu_T$ is given by: 
\begin{equation*}
f_{_X}(\bl)=\sqrt{(2\pi)^N \det(\Sigma_B)} \exp\left(-\dfrac{1}{2}(T(\bl)-\mu_B)'\Sigma_B^{-1}(T(\bl)-\mu_B)\right)
\end{equation*}
One can see that $f_{_X}(\bl)=f_{_Z}(T(\bl))$, where $Z\sim\mathcal{N}(\mu_B,\Sigma_B)$ is a real valued random vector and $f_{_Z}$ is its density with respect to the Lebesgue measure on $\mathbb{R}^N$. For more details about the density function of a normal random vector on a general vector space, see \cite{MorrisEaton}.

For ${\bf a}={\bf 1}$, setting $T$  as the $\mathrm{ilr}$ transformation defined in \cite{distributionsOnSimplex}, we obtain the same density function as for the normal distribution defined in \cite{distributionsOnSimplex}. The authors of the same paper also highlight that the normal law on $\simplex$, corresponding to this density function is equivalent in terms of probabilities to the additive logistic normal law given by Aitchison in \cite{Aitch86}.

	\section{ Appendix}
		
		\medskip
		
		\textcolor{black}{We give a proof for the commutativity of Diagram (\ref{diagrama fa}). First, we recall the necessary theoretical elements from Marsden and Ratiu \cite{mads}}.

		\textcolor{black}{Let $(G,\cdot)$ be a Lie group and $\mathfrak{G}$ its Lie algebra. One can regard the group multiplication as the left action of the group to itself, $g\cdot h=L_g(h)$. The corresponding infinitesimal action is defined by $(\boldsymbol{\xi},g)\rightarrow \boldsymbol{\xi}_L(g):=DL_g(e_{_G}) ~\boldsymbol{\xi}$, where $DL_g(e_{_G})$ is the differential of the left translation computed at the neutral element $e_{_G}\in G$.}
		
		\textcolor{black}{The exponential map $Exp_{_G}:\mathfrak{G}\rightarrow G$ is defined by $$Exp_{_G}(\boldsymbol{\xi}):=\gamma_{_{\boldsymbol{\xi}}}(1),$$ where $\gamma_{_{\boldsymbol{\xi}}}(1)$ is the solution at time $t=1$ for the system of differential equations $\dot\gamma_{_{\boldsymbol{\xi}}}(t)= \boldsymbol{\xi}_L(\gamma_{_{\boldsymbol{\xi}}}(t))$ with initial condition $\gamma_{_{\boldsymbol{\xi}}}(0)=e_{_G}$. }
		
		\textcolor{black}{For the case of Lie group $(\Delta_{N}^0,\plusfa)$, the left translation is defined by the maps composition $$L_{\boldsymbol{\lambda}}^{^{\Delta_{N}^0}}\circ \fa=\fa\circ L_{\boldsymbol{\lambda}}^{^{\mathbb{R}^{N+1}_{>0}}}.$$  }
		
		By differentiation, we obtain the operator equality,
		$$DL_{\boldsymbol{\lambda}}^{^{\Delta_{N}^0}}({\bf e}_{{\bf a}})~\circ~D\fa({\bf{1}})=D\fa(\boldsymbol{\lambda})~\circ~DL_{\boldsymbol{\lambda}}^{^{\mathbb{R}^{N+1}_{>0}}}(\bf{1})$$
		
		and consequently, for any $\boldsymbol{\xi}\in T_{\bf{e_a}}\Delta_{N}^0$, we have the following computation
		
		\begin{align*}
		\boldsymbol{\xi}_L^{^{\Delta_{N}^0}}&=DL_{\boldsymbol{\lambda}}^{^{\Delta_{N}^0}}({\bf e}_{{\bf a}}) \boldsymbol{\xi}=DL_{\boldsymbol{\lambda}}^{^{\Delta_{N}^0}}({\bf e}_{{\bf a}})\left(D\fa({\bf{1}}) {\bf v}_{\boldsymbol{\xi}}\right)\\
		&=D\fa(\boldsymbol{\lambda})\left(DL_{\boldsymbol{\lambda}}^{^{\mathbb{R}^{N+1}_{>0}}}(\bf{1}){\bf v}_{\boldsymbol{\xi}} \right)\\
		&=D\fa(\boldsymbol{\lambda})\left({\bf v}_{\boldsymbol{\xi}}\right)_L^{^{\mathbb{R}^{N+1}_{>0}}},
		\end{align*}
		where ${\bf v}_{\boldsymbol{\xi}}\in T_{\bf 1}\mathbb{R}^{N+1}_{>0}$ such that $\boldsymbol{\xi}=D\fa({\bf{1}}){\bf v}_{\boldsymbol{\xi}}$.This shows that the left invariant vector fields on $(\Delta_{N}^0,\plusfa)$ are the push forward of the left invariant vector fields on $(\mathbb{R}^{N+1}_{>0},\plus)$. Using the definition of the exponential map, we obtain the commutativity of Diagram \ref{diagrama fa}.

	 \bibliographystyle{plain}

	\bibliography{Simplex} 	
		
	\end{document}